\documentclass[12pt]{amsart}
\usepackage[T1]{fontenc}
\usepackage{lmodern, microtype}
\usepackage{amssymb, mathtools}
\usepackage[margin=1in]{geometry}
\usepackage{enumitem, verbatim}
\usepackage[hidelinks]{hyperref}
\newtheorem{theorem}{Theorem}[section]
\newtheorem{lemma}[theorem]{Lemma}
\newtheorem{corollary}[theorem]{Corollary}
\newtheorem{proposition}[theorem]{Proposition}
\newtheorem{question}[theorem]{Question}
\theoremstyle{definition}
\newtheorem{definition}[theorem]{Definition}
\theoremstyle{remark}
\newtheorem{remark}[theorem]{Remark}

\numberwithin{equation}{section}
\DeclareMathOperator{\Lip}{Lip}
\DeclareMathOperator{\conv}{conv}
\DeclareMathOperator{\supp}{supp}

\DeclareMathOperator*{\argmin}{arg\,min}
\newcommand{\R}{\mathbb R}
\newcommand{\norm}[1]{\lVert#1\rVert}
\newcommand{\abs}[1]{\lvert#1\rvert}
\newcommand{\pos}[1]{(#1)^+}
\newcommand{\Pfin}{\mathcal P_{\mathrm{fin}}}
\title[Lipschitz clustering in normed spaces]
{Lipschitz clustering in normed spaces\\and barycentric spaces}
\author{Leonid V. Kovalev}
\date{}
\hypersetup{pdftitle={Lipschitz clustering in normed spaces and barycentric spaces},
pdfauthor={Leonid V. Kovalev}}
\begin{document}
\begin{abstract}
For every normed space $X$ and integers $1\le k<n$, we construct a
Lipschitz retraction from $X(n)$ onto $X(k)$ with constant at most
$2+8n(n-1)$, independently of the dimension of $X$.
Here $X(n)$ consists of the nonempty subsets of $X$ with at most $n$
points, equipped with the Hausdorff metric. For the consecutive case $k=n-1$, a
separate construction gives the sharper asymptotic bound $O(n\log n)$. 
These results extend to barycentric metric spaces and settle some open problems on Lipschitz retractions. 
\end{abstract}
\maketitle

\section{Introduction}\label{sec:introduction}

Clustering reduces a finite configuration to a simpler one by merging
nearby points. Its Lipschitz version asks that small perturbations of
the input produce controlled perturbations of the output. We study
retractions that reduce an $n$-point configuration to at most $k$
representatives, where $1\le k<n$, and fix configurations already
having at most $k$ points. Allowing fewer than $k$ representatives
accommodates simultaneous mergers. Our main setting is that of normed
spaces, with estimates independent of dimension.

For a metric space $X$ and an integer $n\ge1$, write
$X(n)=\{A\subset X:1\le |A|\le n\}$, with the Hausdorff metric
\[
 d_H(A,B)=\max\left\{
 \max_{a\in A}\min_{b\in B}d(a,b),\,
 \max_{b\in B}\min_{a\in A}d(a,b)\right\}.
\]
Following~\cite{KovalevLCP}, $X$ has the \emph{Lipschitz clustering
property} if there is a Lipschitz retraction
$X(n)\to X(n-1)$ for every $n\ge2$. For $1\le k<n$, set
\[
 \Lambda_{n,k}(X)=\inf\{\Lip(r) \ :\  r\colon X(n)\to X(k)
                  \text{ is a Lipschitz retraction}\},
\]
with infimum $+\infty$ if no such retraction exists, and write
$\Lambda_n(X)=\Lambda_{n,n-1}(X)$. If $\Lambda_n(X)<\infty$ for all $n$, 
composing consecutive retractions yields $\Lambda_{n,k}(X)<\infty$ for all $k$,
but with weak quantitative control. Our first theorem avoids this by 
constructing $r\colon X(n)\to X(k)$ directly.

The Lipschitz clustering property was proved for Hilbert spaces in~\cite{Kovalev2016}
and for Hadamard spaces in~\cite{BacakKovalev}. It also holds for
finite-dimensional normed spaces: by John's theorem, a normed space of
dimension $d$ has Banach--Mazur distance at most $\sqrt d$ from $\ell_2^d$,
so the Euclidean case applies, with constants that grow with $d$.
For arbitrary normed spaces, Akofor~\cite{Akofor} obtained Lipschitz
consecutive retractions for $n=2,3$, and consecutive retractions
that are H\"older continuous on bounded sets for general $n$.
He also proved the existence of Lipschitz retractions
$X(n)\to X(1)$ and $X(n)\to X(2)$ for arbitrary
$n>2$~\cite[Theorem 4.11]{Akofor}. The case of general normed spaces was
posed as an open problem in~\cite[Question 3.4]{Kovalev2016}; see
also~\cite[\S4]{KovalevLCP}, and~\cite[Remark 3.5]{BacakKovalev} for
strictly and uniformly convex spaces.

All normed spaces below are real, and completeness is not assumed;
complex normed spaces are covered by regarding them as real spaces. We write $\conv A$ for the convex hull of $A$.

\begin{theorem}\label{thm:direct}
For every normed space $X$ and integers $1\le k<n$, there is a
retraction $r\colon X(n)\to X(k)$ such that
\[
 \Lip(r)\le D_n:=2+8n(n-1),\quad \text{and} \quad
 r(A)\subset\conv A \text{ for all } A\in X(n).
\]
Moreover,
\begin{equation}\label{eq:tree-main-movement}
 d_H(r(A),A)\le(n-1)s_k(A),
\end{equation}
where $s_k(A)$ is the smallest $s\ge0$ such that the graph on $A$
joining all pairs at distance at most $s$ has at most $k$ components
(Definition~\ref{def:cut-height}).
\end{theorem}

For the consecutive case $k=n-1$ a different construction achieves a Lipschitz constant of order $O(n\log n)$.
Here and below, $\log$ denotes the natural logarithm.

\begin{theorem}\label{thm:main}
For every normed space $X$ and every $n\ge2$, there is a retraction
$r\colon X(n)\to X(n-1)$ with
\begin{equation}\label{eq-theorem-main}
 \Lip(r)\le C_n:=1+48n\log(n+1),
 \qquad r(A)\subset\conv A.
\end{equation}
For $n=2$, the midpoint retraction has Lipschitz constant one.
\end{theorem}

One has $C_n<D_n$ exactly when $n\ge19$, so for $k=n-1$
Theorem~\ref{thm:main} improves on Theorem~\ref{thm:direct} from $n=19$ on.
The order $n\log n$ is optimal up to the logarithm: every nonzero normed
space satisfies $\Lambda_n(X)\ge n-3$ (Proposition~\ref{prop:lower}, based
on~\cite[Theorem 5.1]{AkoforKovalev}).

The construction in Theorem~\ref{thm:direct} harmonically interpolates the input vectors on their
single-linkage tree and evaluates the interpolation at height $s_k(A)$. 
The single-linkage hierarchy is the classical hierarchical clustering
obtained from a minimum spanning tree~\cite{GowerRoss,JardineSibson}. Carlsson and M\'emoli~\cite{CarlssonMemoli} 
characterized it axiomatically, in contrast with Kleinberg's impossibility theorem for flat 
clustering~\cite{Kleinberg}, and proved that it is stable under perturbations of
the metric. This stability enters our proof through the fact that all
merge heights move at most twice as fast as the input points. Lipschitz continuity of
algorithms has recently been studied systematically in theoretical
computer science, for example for minimum spanning trees~\cite{KumabeYoshida}.

For later use, set
\begin{equation}\label{eq:combined}
 B_{n,k}=D_n\quad(k\le n-2),\qquad
 B_{n,n-1}=\min\{D_n,C_n\}\quad(n\ge3),\qquad B_{2,1}=1.
\end{equation}
Thus $B_{n,n-1}=D_n$ for $3\le n\le18$ and $B_{n,n-1}=C_n$ for $n\ge19$.
By Theorems~\ref{thm:direct} and~\ref{thm:main}, and the midpoint map
for $n=2$, every normed space admits convex-hull-preserving retractions
$X(n)\to X(k)$ with constant $B_{n,k}$. All these constants are
independent of dimension.

The above results have applications beyond normed spaces. 
Let $\Pfin(X)$ be the finitely supported probability measures on a metric space $X$,
with the $1$-Wasserstein distance $W_1$. A \emph{$K$-Lipschitz
barycenter} (Definition~\ref{def:barycenter}) is a map $\beta\colon \Pfin(X)\to X$ with
$\beta(\delta_x)=x$ and
$d(\beta(\mu),\beta(\nu))\le KW_1(\mu,\nu)$.

\begin{theorem}\label{thm:barycenter-intro}
If a metric space $X$ admits a $K$-Lipschitz barycenter, then
$\Lambda_{n,k}(X)\le K B_{n,k}$, $1\le k<n$. 
Every complete metric space with a conical geodesic bicombing
(Definition~\ref{def:conical}),
including every complete Busemann-convex space and every Hadamard
space, satisfies these bounds with $K=1$.
\end{theorem}

This answers the questions
in~\cite[Remark 3.5]{BacakKovalev} concerning the improvement
of the asymptotic Lipschitz bound for Hadamard spaces and the
existence of such retractions in complete Busemann-convex spaces. For Hadamard
spaces, Theorem~\ref{thm:barycenter-intro} improves the consecutive retraction constant  
$O(n^{3/2})$ of~\cite{BacakKovalev} to
$O(n\log n)$. 

A separate construction (Theorem~\ref{thm:direct-infinity}) gives
$\Lambda_{n,k}(\ell^\infty(\Gamma))\le2n-1$ for $1\le k<n$ and every index
set $\Gamma$, providing the sharp order of growth. Consequently, every absolute
$\lambda$-Lipschitz retract (Definition~\ref{def:alr}) satisfies $\Lambda_{n,k}\le\lambda(2n-1)$. 
 
All Lipschitz constants between finite subset spaces refer to the
Hausdorff metric. For labeled configurations we use the maximum of
the distances between corresponding points.
Section~\ref{sec:direct-infinity} gives efficient and explicit retractions between finite subset spaces of $\ell^\infty$; unfortunately, their output need not lie in the convex hull of the input. This short construction serves as a warm-up for Section~\ref{sec:direct}, where Theorem~\ref{thm:direct} is proved. Section~\ref{sec:consecutive-clustering} proves Theorem~\ref{thm:main}. 
Section~\ref{sec:barycenter} concerns the applications of Theorems~\ref{thm:direct} and~\ref{thm:main} to metric spaces. 
Section~\ref{sec:alr} presents results for absolute Lipschitz retracts, lower bounds, and open questions.

\section{Warm-up: clustering in \texorpdfstring{$\ell^\infty$}{ell-infinity}}\label{sec:direct-infinity}

In this section we construct an explicit $(2n-1)$-Lipschitz retraction $V(n)\to V(k)$ for every $1\le k<n$,
where $V=\ell^\infty(\Gamma)$ and $\Gamma$ is an arbitrary nonempty index set.
The construction is simpler than the one in Section~\ref{sec:direct} and shares some elements with it, such as the following definition. 

\begin{definition}\label{def:cut-height}
Let $A$ be a finite set with a pseudometric $d$. For $s\ge0$, let
$G_s(A)$ be the graph with vertex set $A$ joining $a$ and $b$ whenever
$d(a,b)\le s$. For $k\ge1$, the \emph{cut height} of $A$ is
\begin{equation}\label{eq:tree-scale}
 s_k(A)=\min\{s\ge0:G_s(A)\text{ has at most }k\text{ components}\}.
\end{equation}
\end{definition}

If $d$ is a metric, then $s_k(A)=0$ exactly when $|A|\le k$.

\begin{lemma}\label{lem:cut-height}
\textup{(a)} If $d$ and $d'$ are pseudometrics on a finite set $A$ with
$|d(a,b)-d'(a,b)|\le\eta$ for all $a,b\in A$, then the cut heights of
$A$ for $d$ and for $d'$ differ by at most $\eta$.

\textup{(b)} If $A$ and $B$ are finite subsets of a metric space, then
$|s_k(A)-s_k(B)|\le2d_H(A,B)$.
\end{lemma}

\begin{proof}
(a) Every edge of $G_s(A)$ for $d$ is an edge of $G_{s+\eta}(A)$ for $d'$,
so the latter graph has at most as many components. Hence the cut height
for $d'$ is at most the cut height for $d$ plus $\eta$, and symmetrically.

(b) Let $\varepsilon=d_H(A,B)$ and $s=s_k(A)$, and choose a map
$f\colon A\to B$ with $d(a,f(a))\le\varepsilon$ for all $a\in A$. If $a$
and $a'$ are joined in $G_s(A)$, then $f(a)$ and $f(a')$ are equal or
joined in $G_{s+2\varepsilon}(B)$. Every point of $B$ is within
$\varepsilon$ of some $a\in A$, hence within $2\varepsilon$ of $f(a)$.
Therefore $G_{s+2\varepsilon}(B)$ has at most as many components as
$G_s(A)$, that is, $s_k(B)\le s_k(A)+2\varepsilon$. Interchange $A$ and $B$ to obtain $|s_k(A)-s_k(B)|\le2d_H(A,B)$.
\end{proof}

An isometric embedding of $A$ into $\ell^\infty(\Gamma)$ provides
a collection of $1$-Lipschitz coordinate functions
$x_\gamma\colon A\to\mathbb R$, $\gamma\in\Gamma$, such that
$d(a,b)=\sup_{\gamma\in\Gamma}|x_\gamma(a)-x_\gamma(b)|$ for all $a,b\in A$.
To replace $A$ by a smaller subset of $\ell^\infty(\Gamma)$, we will modify the coordinate functions to make them constant on some subsets of $A$. 

Fix $n\ge 2$; all sets in the rest of this section have at most $n$ points.

For a finite set $A\subset V$ and $s\ge0$, let $\rho_{A,s}$ be the shortest-path
pseudometric on the complete graph with vertex set $A$ and edge lengths $\pos{\norm{a-b}-s}$. Then
\[
 \norm{a-b} -(n-1)s\le\rho_{A,s}(a,b)\le  \norm{a-b},
\]
and the zero-distance classes of $\rho_{A,s}$ are exactly the components of $G_s(A)$.
Define three maps from $A$ to $V$ as follows: 
\[
 L_{A,s}(a)(\gamma)=\min_{z\in A}\bigl(z(\gamma)+\rho_{A,s}(a,z)\bigr),
 \qquad
 U_{A,s}(a)(\gamma)=\max_{z\in A}\bigl(z(\gamma)-\rho_{A,s}(a,z)\bigr),
\]
and $Y_{A,s}(a)=(L_{A,s}(a)+U_{A,s}(a))/2$. Thus, for each coordinate
$\gamma\in\Gamma$, the function $a\mapsto L_{A,s}(a)(\gamma)$ is the largest function on $A$ that is $1$-Lipschitz for $\rho_{A,s}$ and is majorized by coordinate evaluation $a\mapsto a(\gamma)$. Similarly,  $a\mapsto U_{A,s}(a)(\gamma)$ is the smallest such function that is minorized by the evaluation at $\gamma$. Since
$|z(\gamma)-a(\gamma)|\le \norm{a-z}$, the bounds on $\rho_{A,s}$ give
\[
 a-(n-1)s\le L_{A,s}(a)\le a\le U_{A,s}(a)\le a+(n-1)s
\]
coordinatewise. Hence $L_{A,s}(a)$ and $U_{A,s}(a)$ are indeed elements of $V$.  Furthermore, 
$\norm{Y_{A,s}(a)-a} \le (n-1)s/2$,  and $L_{A,0}(a)=U_{A,0}(a)=a$. By the triangle inequality for
$\rho_{A,s}$, all three maps are constant on every zero-distance
class, that is, on every component of $G_s(A)$.

\begin{theorem}\label{thm:direct-infinity}
The formula  $R_{n,k}(A)=\{Y_{A,s_k(A)}(a):a\in A\}$ defines a retraction $V(n)\to V(k)$ with $\Lip(R_{n,k})\le2n-1$.
\end{theorem}
\begin{proof}
The output has at most as many points as $G_{s_k(A)}(A)$ has
components, hence at most $k$. If $|A|\le k$, then $s_k(A)=0$ and
$R_{n,k}(A)=A$.
Given $A, B\in V(n)$, let $\varepsilon=d_H(A,B)$, $s=s_k(A)$ and $t=s_k(B)$. By Lemma~\ref{lem:cut-height}(b), $|s-t|\le2\varepsilon$.

Given $a,z\in A$ and $b,w\in B$ with
$\norm{a-b}\le\varepsilon$ and $\norm{z-w}\le\varepsilon$, put
$c=2\varepsilon+s-t\ge0$.
If $a\ne z$, choose a simple path
$a=a_0,\ldots,a_q=z$ realizing $\rho_{A,s}(a,z)$, where
$q\le |A|-1\le n-1$. Choose $b_j\in B$ with
$\norm{a_j-b_j}\le\varepsilon$, taking $b_0=b$ and $b_q=w$.
For each edge,
\[
 \pos{\norm{b_j-b_{j-1}}-t}
 \le \pos{\norm{a_j-a_{j-1}}-s}+c.
\]
Summing gives
\[
 \rho_{B,t}(b,w)
 \le \rho_{A,s}(a,z)+(n-1)(2\varepsilon+s-t).
\]
If $a=z$, the same inequality follows from
\[
 \rho_{B,t}(b,w)
 \le \pos{2\varepsilon-t}
 \le c\le(n-1)c,
\]
since $\rho_{A,s}(a,z)=0$.
Interchanging $(A,s,a,z)$ and $(B,t,b,w)$ gives
\[
 \rho_{A,s}(a,z)
 \le \rho_{B,t}(b,w)+(n-1)(2\varepsilon+t-s).
\]
Since $\norm{a-b} \le\varepsilon$, these comparisons give coordinatewise inequalities
\begin{equation*}\begin{split}
 L_{B,t}(b)-L_{A,s}(a) & \le\varepsilon+(n-1)(2\varepsilon+s-t),
 \\
 U_{B,t}(b)-U_{A,s}(a) & \le\varepsilon+(n-1)(2\varepsilon+t-s).
\end{split}\end{equation*}
Thus $Y_{B,t}(b)-Y_{A,s}(a)\le (2n-1)\varepsilon$ coordinatewise. By symmetry, 
$\norm{Y_{A,s}(a)-Y_{B,t}(b)} \le(2n-1)\varepsilon$, and the theorem follows.
\end{proof}

\section{Clustering by harmonic interpolation: proof of Theorem~\ref{thm:direct}}\label{sec:direct}

\subsection{Single-linkage trees and harmonic interpolation}

For a finite metric space $(A,d)$, define
\begin{equation}\label{eq:tree-ultrametric}
 u_A(a,b)=\min_{a=a_0,\ldots,a_q=b}\ \max_{1\le j\le q}
                       d(a_{j-1},a_j),\qquad u_A(a,a)=0.
\end{equation}
This is the single-linkage ultrametric
\cite{CarlssonMemoli,GowerRoss,JardineSibson}. A  path realizing $u_A$ may be
taken simple, so if $|A|=m$, then
\begin{equation}\label{eq:tree-distortion}
 u_A(a,b)\le d(a,b)\le(m-1)u_A(a,b).
\end{equation}
The second inequality follows by summing the original distances
along a simple minimizing path with at most $m-1$ edges.

The components of the graph $G_s(A)$ of Definition~\ref{def:cut-height}
are the equivalence classes of the relation $u_A(a,b)\le s$.

The nested components of graphs $G_s(A)$ form a rooted metric tree $T_A$. Its leaves
are the points of $A$, at height zero; the height $h(v)$ of a branching
vertex $v$ is the threshold $s$ at which its children merge. An edge from
a child $v$ to its parent $p(v)$ has length
$\ell_v=h(p(v))-h(v)$.
For two leaves $a,b$, their least common ancestor has height $u_A(a,b)$,
and their tree distance is
\begin{equation}\label{eq:tree-tree-distance}
 d_{T_A}(a,b)=2u_A(a,b).
\end{equation}

For a finite rooted metric tree $T$ with leaves $a_1,\ldots,a_m$, a
function $G$ on its vertices, extended affinely along the edges, has
\emph{Laplacian}
\[
 \Delta G(v)=\sum_{w\sim v}\frac{G(v)-G(w)}{\ell_{vw}}
\]
at each nonleaf vertex $v$. Thus $-\Delta G(v)$ is the sum of the slopes
of $G$ along the edges leaving $v$. We call $G$ \emph{harmonic} if
$\Delta G=0$ at every nonleaf
vertex. This definition applies equally to vector-valued functions. Given boundary (vector) values $x_1,\ldots,x_m\in X$,
the \emph{harmonic interpolation} $F\colon T\to X$ is affine on each edge,
takes value $x_i$ at $a_i$, and satisfies
\begin{equation}\label{eq:tree-harmonic}
 \Delta F(v)=\sum_{w\sim v}\frac{F(v)-F(w)}{\ell_{vw}}=0
\end{equation}
at every nonleaf vertex $v$, including the root. By~\eqref{eq:tree-harmonic}, $F(v)$ is a weighted average of adjacent values, which yields the maximum principle in the scalar-valued case. The maximum principle
gives uniqueness of solution, and therefore its existence, since the matrix of~\eqref{eq:tree-harmonic} is square~\cite[\S2.1]{LyonsPeres}. For a leaf $a_i$, let $\omega_\cdot(i)$ be
the scalar solution equal to $1$ at $a_i$ and to $0$ at the other leaves.
Thus $\omega_v(i)$ is the \emph{harmonic measure} of $a_i$ with respect to $v$. The maximum principle gives
$\omega_v(i)\ge0$, and we have $\sum_i\omega_v(i)=1$ because constants are
harmonic. We obtain the unique solution of~\eqref{eq:tree-harmonic}, in any vector space $X$, via the formula 
\begin{equation}\label{eq:harmonic-measure-formula-F}
 F(v)=\sum_i\omega_v(i)x_i\in\conv\{x_1,\ldots,x_m\}.
\end{equation}

\begin{lemma}[Edge slopes]\label{lem:tree-slopes}
Let $T$ be a finite rooted tree with $m\ge2$ leaves, with heights as
above: leaves at height zero and edge lengths given by height
differences. Let $F$ be the harmonic interpolation of boundary vectors
satisfying
\[
 \norm{x_i-x_j}\le L\,d_T(a_i,a_j).
\]
For a nonroot vertex $v$, let
\[
 D_vF=\frac{F(v)-F(p(v))}{\ell_v}
\]
be the slope of $F$ on the edge from $p(v)$ down to $v$, and let $C_v$
be the set of leaves descending from $v$. Put $q_i=D_{a_i}F$. If $v$ is
an ancestor of $a_i$, then
\begin{equation}\label{eq:tree-ancestor}
 \norm{F(v)-x_i}\le2Lh(v).
\end{equation}
Moreover, $\norm{q_i}\le2L$,
\begin{equation}\label{eq:tree-flow}
 D_vF=\sum_{i\in C_v}q_i,\qquad \sum_iq_i=0,
\end{equation}
and $\norm{D_vF}\le2L\lfloor m/2\rfloor$ for every nonroot vertex $v$.
\end{lemma}
\begin{proof}
Fix a leaf $a_i$. Write $v\wedge a_i$ for the least common ancestor
of $v$ and $a_i$, and consider
the function $g_i(v)=2h(v\wedge a_i)$. 
On the root-to-$a_i$ path this function has slope $2$ in the upward direction;
it is constant on every branch off that path. Hence $\Delta g_i=0$ at
all nonleaf vertices except the root, where $\Delta g_i=2$. Its boundary
values are $g_i(a_j)=d_T(a_i,a_j)$.

For every functional $\varphi\in X^*$ with $\norm{\varphi}\le1$, the function
$G=\varphi(F-x_i)-Lg_i$ satisfies $\Delta G\le0$ at the nonleaf vertices, that
is, $G(w)$ is at most the weighted average of the adjacent values.
Hence $G$ attains its maximum at a leaf, where $G\le0$ by the hypothesis
on the boundary vectors. Thus $\varphi(F-x_i)\le Lg_i$, and taking the
supremum over $\varphi$ proves~\eqref{eq:tree-ancestor}. For $v=p(a_i)$ it gives
$\norm{q_i}=\norm{x_i-F(p(a_i))}/h(p(a_i))\le2L$.

In $\Delta F(v)$, the edge above $v$ contributes $D_vF$, and each edge
below $v$ contributes minus its slope. Thus \eqref{eq:tree-harmonic} says
that at every nonleaf vertex $v$ other than the root, $D_vF$ is the sum
of the slopes below $v$, and that the slopes below the root sum to zero.
Induction from the leaves proves the first identity in~\eqref{eq:tree-flow},
and the condition at the root gives the second. Hence
\[
 \norm{D_vF}\le2L\min\{|C_v|,m-|C_v|\}\le2L\lfloor m/2\rfloor.
\]
The same arguments apply when some nonroot vertices have a single child (which may happen when
we need to divide some edges). 
\end{proof}

\subsection{Definition of retraction in Theorem~\ref{thm:direct}} Given a nonempty finite set $A\subset X$, where $X$ is a normed space, let $F_A$ 
be the harmonic interpolation on the tree $T_A$ with boundary value $a$ at each leaf $a\in A$. 
For $0\le s\le h(\mathrm{root})$, let $p_a(s)$
be the ancestor of $a$ at height $s$, and define
\begin{equation}\label{eq:tree-definition}
 r_k(A)=\{F_A(p_a(s_k(A))):a\in A\}.
\end{equation}
For a singleton this means the same singleton. The formula defines $r_k$
on all nonempty finite subsets of $X$; the retraction in
Theorem~\ref{thm:direct} is its restriction to $X(n)$.
There are at most $k$ distinct cut points $p_a(s_k(A))$, so the
output has at most $k$ points. If $|A|\le k$, then $s_k(A)=0$ and $r_k(A)=A$. 
Convex-hull preservation follows from the formula~\eqref{eq:harmonic-measure-formula-F} expressing $F$ in terms of harmonic measure. 
By~\eqref{eq:tree-distortion}--\eqref{eq:tree-tree-distance} we may apply Lemma~\ref{lem:tree-slopes}
with $L=(n-1)/2$ when $A\in X(n)$; then \eqref{eq:tree-ancestor} proves \eqref{eq:tree-main-movement}.
It remains to prove the Lipschitz property of $r_k$.

\subsection{Comparing two harmonic interpolations}

We now compare harmonic interpolations on one tree for two choices of
heights and boundary vectors. The difference is controlled by the
changes of vertex heights; edge lengths, being differences of heights,
do not enter separately.

Let $T$ be a finite rooted tree with leaves $a_1,\ldots,a_m$, in which
nonroot vertices may have a single child. A \emph{height function} on
$T$ is a function $h$ on its vertices that vanishes at the leaves and
satisfies $h(v)\le h(p(v))$ for every nonroot vertex $v$, so that the
edge above $v$ has length $\ell_v=h(p(v))-h(v)\ge0$. Lengths zero occur
in the proof of Lemma~\ref{lem:tree-tuple}, at the finitely many times
where the single-linkage tree changes or points merge. If some lengths
vanish, the harmonic interpolation of boundary vectors $x_i$ is defined
on the tree obtained by contracting the edges of length zero, provided
$x_i=x_j$ whenever $a_i$ and $a_j$ are contracted to the same vertex;
each vertex of $T$ receives the value at its image. Leaf slopes are then
defined as follows. If $\ell_{a_i}>0$, put $q_i=D_{a_i}F$. Otherwise let
$u$ be the highest ancestor of $a_i$ with $h(u)=0$, whose subtree is
contracted to a leaf, and put $q_i=D_uF/|C_u|$, or $q_i=0$ if $u$ is the
root. The identities~\eqref{eq:tree-flow} in the contracted tree give
\begin{equation}\label{eq:tree-flow-weak}
 \sum_{i\in C_v}q_i=D_vF\quad\text{whenever }\ell_v>0,
 \qquad \sum_iq_i=0.
\end{equation}

\begin{lemma}[Comparison of harmonic interpolants]\label{lem:tree-comparison}
Let $h$ and $h'$ be height functions on $T$ such that all edge lengths
$\ell_v$ for $h$ are positive. Let $F$ and $F'$ be the harmonic
interpolations of boundary vectors $x_j$ for $h$ and $x'_j$ for $h'$,
let $q'_i$ be the leaf slopes of $F'$, and let $\omega$ be harmonic
measure for $h$. Then, for every vertex $v$,
\begin{equation}\label{eq:tree-variation}
 F'(v)-F(v)=\sum_j\omega_v(j)(x'_j-x_j)+
 \sum_iq'_i\Bigl(
       \sum_j\omega_v(j) \eta(a_i\wedge a_j)
                  -\eta(a_i\wedge v)\Bigr)
\end{equation}
where $\eta  =h'-h$.
Consequently, if $\norm{x'_j-x_j}\le\beta$ and $\norm{q'_i}\le Q$ for all
$i,j$, and $|\eta(w)|\le\delta$ for all vertices $w$, then
\begin{equation}\label{eq:tree-variation-bound}
 \norm{F'(v)-F(v)}\le\beta+2mQ\delta.
\end{equation}
\end{lemma}

\begin{proof}
Let $\ell'$ be the edge lengths for $h'$, and put $\eta=h'-h$. For every
nonroot vertex $v$, put $D_vF'=\sum_{i\in C_v}q'_i$. By
\eqref{eq:tree-flow-weak}, this is the slope of $F'$ when $\ell'_v>0$, so
in all cases
\[
 F'(v)-F'(p(v))=\ell'_v\,D_vF',
\]
both sides vanishing when $\ell'_v=0$. Consider
\[
 \Theta(v)=\sum_iq'_i\,\eta(a_i\wedge v),\qquad G=F'+\Theta.
\]
If $i\in C_v$, then $a_i\wedge v=v$ and $a_i\wedge p(v)=p(v)$; otherwise
$a_i\wedge v=a_i\wedge p(v)$. Hence
\[
 \Theta(v)-\Theta(p(v))=D_vF'\bigl(\eta(v)-\eta(p(v))\bigr)
 =(\ell_v-\ell'_v)\,D_vF',
\]
and therefore $G(v)-G(p(v))=\ell_v\,D_vF'$. 

By their definition as sums of leaf slopes, the quantities $D_vF'$
satisfy
\[
 D_vF'=\sum_{p(w)=v}D_wF'
\]
at every nonroot nonleaf vertex $v$. At the root, the sum over its
children is $\sum_iq'_i=0$. These identities hold even when some
$\ell'_v$ vanish. Since
\[
 \frac{G(v)-G(p(v))}{\ell_v}=D_vF',
\]
they are precisely the harmonicity equations for $G$ with respect
to the positive edge lengths $\ell$. 

At the leaves we have $G(a_j) = x'_j+\Theta(a_j)$, so
\[
 F'(v)+\Theta(v)=\sum_j\omega_v(j)\bigl(x'_j+\Theta(a_j)\bigr)
\]
for every vertex $v$. Subtracting $F(v)=\sum_j\omega_v(j)x_j$
gives~\eqref{eq:tree-variation}. The estimate~\eqref{eq:tree-variation-bound} follows from~\eqref{eq:tree-variation} via the triangle inequality.
\end{proof}

\subsection{Moving the points of a finite set}

\begin{lemma}[Moving points]\label{lem:tree-tuple}
Let $A\subset X$ be a set with $m$ points, and let $f\colon A\to X$
satisfy $\norm{f(a)-a}\le\varepsilon$ for all $a\in A$. Then, for every
$k\ge1$, the retraction~\eqref{eq:tree-definition} satisfies 
\begin{equation}\label{eq:tree-tuple}
 d_H\bigl(r_k(A),r_k(f(A))\bigr)\le\bigl(1+4m(m-1)\bigr)\varepsilon.
\end{equation}
\end{lemma}
\begin{proof}
The assertion is immediate for $m=1$, so let $m\ge2$, and put
$K=1+4m(m-1)$. For $a,b\in A$ and $0\le t\le1$, interpolate the
boundary data and the distances:
\begin{align*}
 z_a(t)&=(1-t)a+tf(a),\\
 d_{ab}(t)&=(1-t)\norm{a-b}+t\norm{f(a)-f(b)}.
\end{align*}
Then, for $t,\tau\in[0,1]$,
\[
 \norm{z_a(t)-z_a(\tau)}\le\varepsilon|t-\tau|,\qquad
 |d_{ab}(t)-d_{ab}(\tau)|\le2\varepsilon|t-\tau|,\qquad
 \norm{z_a(t)-z_b(t)}\le d_{ab}(t).
\]
For $t<1$, $d(t)$ is a metric on $A$. For $t=1$, we identify points of
$A$ with the same image under $f$; then $d(1)$ becomes the metric of
$f(A)$, and $z_a(1)=f(a)$. Let $T(t)$ be the single-linkage tree of
$(A,d(t))$, with this identification when $t=1$, let $F_t$ be the
harmonic interpolation of the vectors $z_a(t)$ on it, and let $s(t)$ be
the cut height $s_k(A)$ for the pseudometric $d(t)$. For $a\in A$, let
$p_a(s)$ be the point at height $s$ on the path from the leaf $a$ to the
root, and put
\[
 \Phi_a(t)=F_t\bigl(p_a(s(t))\bigr).
\]
By~\eqref{eq:tree-definition}, $\{\Phi_a(0)\}_{a\in A}=r_k(A)$ and
$\{\Phi_a(1)\}_{a\in A}=r_k(f(A))$. Fix $a$. It suffices to find a
subdivision of $[0,1]$ such that
\begin{equation}\label{eq:tree-step}
 \norm{\Phi_a(t)-\Phi_a(\tau)}\le K\varepsilon|t-\tau|
\end{equation}
whenever $t$ and $\tau$ lie in the closure of one of its intervals;
adding these estimates gives $\norm{\Phi_a(1)-\Phi_a(0)}\le K\varepsilon$.

Choose the subdivision so that on each open interval $J$ the ordering of
the finitely many affine functions $d_{bc}$ is constant. Then the tree
$T(t)$ is the same for all $t\in J$, with ties that persist on $J$
treated as simultaneous mergers; its leaves are the points of $A$, and
each vertex height is one of the functions $d_{bc}$ on $J$. The cut
height $s(t)$ is identically zero or equal to one of these functions on
$J$; it is continuous on $[0,1]$, since
$|s(t)-s(t')|\le2\varepsilon|t-t'|$ by Lemma~\ref{lem:cut-height}(a). Subdivide, if necessary, the edge containing
$p_a(s(t))$, $t\in J$, so that this point becomes a vertex $v_a$. Denote
the resulting tree by $T_J$, and for $t$ in the closure $\overline{J}$ let
$h_t$ be the height function on $T_J$ given by these affine functions,
so that $v_a$ has height $s(t)$.

For $t\in\overline{J}$, contracting the edges of $T_J$ of length zero for
$h_t$ gives $T(t)$, with $v_a$ going to $p_a(s(t))$. Indeed, for
$b,c\in A$, the height $h_t(b\wedge c)$ equals the single-linkage
ultrametric $u_{d(t)}(b,c)$ when $t\in J$, both are continuous in $t$,
and a rooted tree with leaves at height zero and positive edge lengths
is determined, up to vertices with a single child, by the heights of the
least common ancestors of its leaves. Thus $F_t$ is the harmonic
interpolation on $T_J$ for $h_t$, and $\Phi_a(t)=F_t(v_a)$.

Let $t\in\overline{J}$ and $\tau\in J$. Apply Lemma~\ref{lem:tree-comparison}
to $T_J$, which has $m$ leaves, with $h=h_\tau$, $h'=h_t$, and boundary
vectors $z_b(\tau)$ and $z_b(t)$. Vertex heights are affine with slopes
at most $2\varepsilon$ in absolute value, so we may take
$\beta=\varepsilon|t-\tau|$ and $\delta=2\varepsilon|t-\tau|$. If $T(t)$
has $m_t\ge2$ leaves, the triangle inequality for $d(t)$, together with
\eqref{eq:tree-distortion} and~\eqref{eq:tree-tree-distance}, gives
\[
 \norm{z_b(t)-z_c(t)}\le\frac{m_t-1}{2}\,d_{T(t)}(b,c),
\]
so by Lemma~\ref{lem:tree-slopes} the leaf slopes of $F_t$ have norm at
most $m_t-1\le m-1$; if $m_t=1$, they vanish. Hence
\eqref{eq:tree-variation-bound} with $Q=m-1$ gives~\eqref{eq:tree-step}.

If $t$ and $\tau$ are the endpoints of $J$, inequality~\eqref{eq:tree-step} follows by applying the preceding paragraph 
to pairs of points where one is an endpoint of $J$ and the other is its midpoint. 
\end{proof}

\begin{lemma}[Common image]\label{lem:common-image}
Let $A,B\subset X$ be finite sets with $d_H(A,B)\le\varepsilon$. There
are maps $f\colon A\to X$ and $g\colon B\to X$ such that $f(A)=g(B)$
and for all $a\in A$ and $b\in B$ we have $\norm{f(a)-a}\le\varepsilon$ and $\norm{g(b)-b}\le\varepsilon$.
\end{lemma}

\begin{proof}
Consider relations $R\subset A\times B$ such that $aRb$ implies
$\norm{a-b}\le\varepsilon$, every point of $A$ is related to some point
of $B$, and every point of $B$ is related to some point of $A$. 
Choose such an $R$ with as few pairs as possible. If $aRb$, then 
$a$ and $b$ cannot both be related to other points, since otherwise
the pair $(a,b)$ could be removed from $R$. For $aRb$, put $c(a,b)=a$ if $a$ is related to more than one point, and $c(a, b)=b$ otherwise. 
By the preceding observation, $c(a,b)=b$ whenever $b$ is related to more
than one point. Hence $c(a,b)=c(a,b')$ whenever $aRb$ and $aRb'$, and
$c(a,b)=c(a',b)$ whenever $aRb$ and $a'Rb$. Thus $f(a)=c(a,b)$ and
$g(b)=c(a,b)$, for $aRb$, define maps on $A$ and on $B$, both with image
$C=\{c(a,b):aRb\}$. They move points by at most $\varepsilon$, since
$c(a,b)\in\{a,b\}$.
\end{proof}

\begin{proof}[Proof of Theorem~\ref{thm:direct}]
Let $A,B\in X(n)$ and $\varepsilon=d_H(A,B)$. Let $f$, $g$, and $C=f(A)=g(B)$ be as
in Lemma~\ref{lem:common-image}. Since $|A|,|B|\le n$,
Lemma~\ref{lem:tree-tuple} applied to $f$ and to $g$ gives
\[
 d_H(r_k(A),r_k(B))\le d_H(r_k(A),r_k(C))+d_H(r_k(C),r_k(B))
 \le2\bigl(1+4n(n-1)\bigr)\varepsilon=D_n\varepsilon.
\]
The retraction, convex-hull, and movement assertions were proved
after \eqref{eq:tree-definition}.
\end{proof}

\section{Clustering by collision: proof of Theorem~\ref{thm:main}}\label{sec:consecutive-clustering}

We first outline the construction of a retraction $X(n)\to X(n-1)$. Embed $X$
linearly and isometrically into $\ell^\infty(B_{X^*})$. Given an $n$-point
set $A$, apply the proximal map of the convex function~\eqref{eq:define-Psi} coordinatewise at a scale chosen so that the closest pair(s) of points 
collide in every coordinate at once. The output is $1$-Lipschitz with respect to the input and $(n-1)$-Lipschitz with respect to the scale; 
however, it need not lie in the image of $X$.  
Therefore, each point $x$ is replaced by a convex combination of the points of $A$, with weights decreasing in the distance to $x$. 
This step contributes the logarithmic factor in~\eqref{eq-theorem-main}. 

\subsection{Separation and matching} This short subsection borrows from~\cite[Lemmas 2.5 and 2.6]{Akofor}. 

\begin{definition}\label{def:separation}
The \emph{separation} of $A\in X(n)$ is
\begin{equation}\label{eq:separation}
\delta_n(A)=
\begin{cases}
\displaystyle\min_{\substack{a,b\in A, \ a\ne b}}d(a,b),&|A|=n,\\[4pt]
0,&|A|<n.
\end{cases}
\end{equation}
\end{definition}

\begin{lemma}\label{lem:matching}
Let $A,B\in X(n)$ and $\varepsilon=d_H(A,B)$. Then (a) $|\delta_n(A)-\delta_n(B)|\le2\varepsilon$; 
(b) if $\max\{\delta_n(A),\delta_n(B)\}>2\varepsilon$, then both sets have $n$ points and admit enumerations
$A=\{a_1,\ldots,a_n\}$ and $B=\{b_1,\ldots,b_n\}$ such that
$d(a_i,b_i)\le\varepsilon$ for all $i$.
\end{lemma}

\begin{proof} If $\max\{\delta_n(A),\delta_n(B)\}\le 2\varepsilon$, then (a) is trivial. Suppose that $\delta_n(A)>2\varepsilon$. For each $a\in A$, choose
a point of $B$ within distance $\varepsilon$. Distinct points of $A$
must have distinct choices, so $|B|=n$ and these choices give a bijection required by (b). Under this bijection every pairwise distance changes by at most $2\varepsilon$, which implies (a). 
\end{proof}

\subsection{The proximal map} The map $P_t$ used below is the scalar, equal-weight case of
the convex-fusion clustering objective studied
in~\cite{HockingEtAl}. 
To define it, fix $n\ge2$ and set
\begin{equation}\label{eq:define-Psi}
\Psi(u)=\sum_{1\le i<j\le n}|u_i-u_j|,
\qquad u\in\R^n.    
\end{equation}
For $t\ge0$ and $a\in\R^n$, define
\begin{equation}\label{eq:prox}
P_t(a)=\argmin_{u\in\R^n}
\left\{\frac12\sum_{i=1}^n(u_i-a_i)^2+t\Psi(u)\right\}.
\end{equation}
The objective function in~\eqref{eq:prox} is coercive and strictly convex, so its minimizer exists
and is unique. In particular, $P_0$ is the identity. Permuting the
coordinates of $a$ permutes those of $P_t(a)$ in the same way.

\begin{lemma}\label{lem:scalar}
For all $a,b\in\R^n$ and $t,s\ge0$,
\begin{align}
\norm{P_t(a)-P_s(b)}_\infty
&\le\norm{a-b}_\infty+(n-1)|t-s|,
\label{eq:jointlip}\\
\norm{P_t(a)-a}_\infty&\le(n-1)t.
\label{eq:movement}
\end{align}
Moreover, for every $i,j$,
\begin{equation}\label{eq:collapse}
|a_i-a_j|\le2t
\quad\Longrightarrow\quad
(P_t(a))_i=(P_t(a))_j.
\end{equation}
\end{lemma}

\begin{proof}
Write $u=P_t(a)$, and let $\Phi$ be the objective in~\eqref{eq:prox}.
For $w\in\R^n$, the one-sided directional derivative
\[
\Phi'(u;w)=\lim_{\theta\to0^+}\frac{\Phi(u+\theta w)-\Phi(u)}{\theta}
\]
exists since $\Phi$ is convex, and $\Phi'(u;w)\ge0$ since $u$ is a
minimizer. 

For~\eqref{eq:collapse}, suppose that $u_i<u_j$, and let $w=e_i-e_j$.
The term $|u_i-u_j|$ contributes $-2$ to $\Psi'(u;w)$. For
$l\notin\{i,j\}$, the terms $|u_i-u_l|$ and $|u_l-u_j|$ contribute
$\pm1$ each, with $+1$ only if $u_i\ge u_l$, respectively $u_l\ge u_j$.
Since $u_i<u_j$, at most one of them is $+1$. Hence $\Psi'(u;w)\le-2$, and
\[
0\le\Phi'(u;w)\le(u_i-a_i)-(u_j-a_j)-2t,
\]
that is, $a_j-a_i\ge2t+(u_j-u_i)>2t$. Interchanging $i$ and $j$
proves~\eqref{eq:collapse}.

For~\eqref{eq:jointlip}, let $v=P_s(b)$, let $\tilde\Phi$ be the
objective defining $v$, and put
\[
M=\max_i(u_i-v_i),\qquad E=\{i:u_i-v_i=M\}.
\]
Let $\boldsymbol 1_E$ be the indicator vector of $E$. The quadratic terms
of $\Phi'(u;-\boldsymbol 1_E)$ and $\tilde\Phi'(v;\boldsymbol 1_E)$ add up
to $-\sum_{i\in E}\bigl(M-(a_i-b_i)\bigr)$. In $\Psi$, only the terms
with $i\in E$ and $j\notin E$ contribute. Such a term $|u_i-u_j|$
contributes $t$ times $\pm1$, with $+1$ only if $u_i\le u_j$, and the
corresponding term $|v_i-v_j|$ contributes $s$ times $\pm1$, with $+1$
only if $v_i\ge v_j$. Since $u_i-u_j>v_i-v_j$, at most one of them is
$+1$, so together they contribute at most $|t-s|$. There are at most
$|E|(n-1)$ such pairs. Adding $\Phi'(u;-\boldsymbol 1_E)\ge0$ and
$\tilde\Phi'(v;\boldsymbol 1_E)\ge0$ therefore gives
\[
|E|M\le\sum_{i\in E}(a_i-b_i)+|E|(n-1)|t-s|
\le|E|\bigl(\norm{a-b}_\infty+(n-1)|t-s|\bigr).
\]
Interchanging $(a,t)$ and $(b,s)$ proves~\eqref{eq:jointlip}. Finally,
\eqref{eq:movement} is the case $b=a$, $s=0$ of~\eqref{eq:jointlip},
since $P_0$ is the identity.
\end{proof}

\subsection{Coordinatewise collision}

Let $\Gamma$ be any nonempty index set and $V=\ell^\infty(\Gamma)$. For the rest of the section the vector norms are $\ell^\infty$ norms. 
For an ordered tuple $x=(x_1,\ldots,x_n)\in V^n$, define
$\mathcal P_t(x)=(y_1,\ldots,y_n)$ coordinatewise by
\begin{equation}\label{eq:coordinate}
\bigl(y_1(\gamma),\ldots,y_n(\gamma)\bigr)
=P_t\bigl(x_1(\gamma),\ldots,x_n(\gamma)\bigr),
\qquad \gamma\in\Gamma.
\end{equation}
By~\eqref{eq:movement}, the $y_i$ are bounded functions. Lemma~\ref{lem:scalar}
implies
\begin{align}
\max_i\norm{y_i-x_i} &\le(n-1)t,
\label{eq:vector-movement}\\
\max_i\norm{\mathcal P_t(x)_i-\mathcal P_s(x')_i} 
&\le\max_i\norm{x_i-x'_i} +(n-1)|t-s|.
\label{eq:vector-joint}
\end{align}
Furthermore,~\eqref{eq:collapse}  gives the collision property
\begin{equation}\label{eq:vector-collision}
\norm{x_i-x_j} \le2t
\quad\Longrightarrow\quad y_i=y_j.
\end{equation}

We will use the above properties of $\mathcal P_t$ to obtain convex-hull-preserving 
retractions $X(n)\to X(n-1)$ in a general normed space $X$. Toward that, recall that 
$X$ embeds linearly and isometrically into $V$ for some index set $\Gamma$, e.g.,
\(\Gamma=B_{X^*}\). Thus, we consider $X$ to be a (not necessarily closed) subspace of $V$.  

\subsection{The averaging map} Fix \(n\ge2\). The averaging construction uses the parameters
\[
c=\frac{n-1}{2},\quad b=\frac{3n-1}{2},\quad
\lambda=\frac{\log(n+1)}{n},
\]
and the radial weight function 
\[
\varphi(r)=e^{\lambda\pos{b-r}}-1, \quad r\ge0.
\]
This function is nonincreasing, vanishes on \([b,\infty)\), and satisfies
\(\varphi(c)=e^{\lambda n}-1=n\). Since \(\varphi'=-\lambda(\varphi+1)\) on
\([0,b)\) and \(\varphi\) is constant on \([b,\infty)\), the mean value
theorem gives
\begin{equation}\label{eq:profile}
\abs{\varphi(r)-\varphi(r')}\le
\lambda\bigl(\max\{\varphi(r),\varphi(r')\}+1\bigr)\abs{r-r'},
\qquad r,r'\ge0 .
\end{equation}

For a tuple \(x=(x_1,\ldots,x_n)\in X^n\) and a scale \(s>0\), set
$(y_1,\ldots,y_n)=\mathcal P_{s/2}(x_1,\ldots,x_n)$, and define
\begin{equation}\label{eq:transfer-weights}
w_{ij}=\varphi\Bigl(\frac{\norm{y_i-x_j} }{s}\Bigr),\quad
W_i=\sum_{j=1}^nw_{ij},\quad p_{ij}=\frac{w_{ij}}{W_i},\quad
T_s(x)_i=\sum_{j=1}^np_{ij}x_j.
\end{equation}
Here $W_i>0$ by \eqref{eq:transfer-mass} below, so division by $W_i$ makes sense. 
The point \(T_s(x)_i\) is an average of the points \(x_j\), weighted according
to the distance of \(x_j\) from \(y_i\). 

\begin{lemma}\label{lem:local-transfer}
The map \(T_s\) commutes with permutations of the labels and satisfies
\begin{align}
T_s(x)_i&\in\conv\{x_1,\ldots,x_n\},\label{eq:transfer-convex}\\
\norm{T_s(x)_i-x_i}&\le(2n-1)s,\label{eq:transfer-move}\\
\norm{x_i-x_j}\le s&\ \Longrightarrow\
T_s(x)_i=T_s(x)_j.\label{eq:transfer-collision}
\end{align}
If
$\max_i\norm{x_i-x'_i}\le\varepsilon$, 
$|s-s'|\le2\varepsilon$, and $\max\{s,s'\}>2\varepsilon$, then
\begin{equation}\label{eq:transfer-lip}
\max_i\norm{T_s(x)_i-T_{s'}(x')_i}\le C_n\varepsilon.
\end{equation}
\end{lemma}

\begin{proof}
By~\eqref{eq:vector-movement}, \(\norm{y_i-x_i} \le cs\). Since
\(\varphi\) is nonincreasing,
\begin{equation}\label{eq:transfer-mass}
W_i\ge w_{ii}\ge\varphi(c)=n. 
\end{equation}
Thus $T_s$ is well-defined. Permutation equivariance and the convex hull property~\eqref{eq:transfer-convex} 
hold by the formula for $T_s$. If \(w_{ij}>0\), then $\norm{y_i-x_j}< bs$, hence 
\begin{equation}\label{eq:transfer-support}
\norm{x_j-x_i}\le\norm{x_j-y_i} +\norm{y_i-x_i} 
<(b+c)s=(2n-1)s,
\end{equation}
which proves~\eqref{eq:transfer-move}. If \(\norm{x_i-x_j}\le s\), then
\(y_i=y_j\) by~\eqref{eq:vector-collision}, and~\eqref{eq:transfer-collision} follows.

For~\eqref{eq:transfer-lip}, we may fix \(i\) and assume \(s\ge s'\). Then
\(s>2\varepsilon\) and \(s'\ge s-2\varepsilon>0\). Below we use primes for the
quantities defined from \(x'\) and \(s'\).

\smallskip\noindent
\textsc{Step 1: comparing distances.}
Let \(R_j=\norm{y_i-x_j} \) and \(R'_j=\norm{y'_i-x'_j} \).
By~\eqref{eq:vector-joint}, \(\norm{y_i-y'_i} \le\varepsilon+c(s-s')\),
hence \(\abs{R_j-R'_j}\le2\varepsilon+c(s-s')\). Since \(\varphi\) is
constant on \([b,\infty)\), we have \(w_{ij}=\varphi(m_j)\) and
\(w'_{ij}=\varphi(m'_j)\), where
\[
m_j=\min\{R_j/s,\,b\},\qquad m'_j=\min\{R'_j/s',\,b\}.
\]
Replacing \(s'\) by \(s\) in \(m'_j\) decreases it by at most
\(b(s-s')/s\): if \(R'_j\le bs'\), the decrease is
\(R'_j(1/s'-1/s)\le b(s-s')/s\); otherwise \(m'_j=b\) and
\(\min\{R'_j/s,b\}\ge bs'/s\). Since \(t\mapsto\min\{t,b\}\) is
\(1\)-Lipschitz, it follows that
\begin{equation}\label{eq:transfer-rho}
\abs{m_j-m'_j}
\le\frac{2\varepsilon+(b+c)(s-s')}{s}
\le\frac{4n\varepsilon}{s},
\end{equation}
using \(b+c=2n-1\) and \(s-s'\le2\varepsilon\).

\smallskip\noindent
\textsc{Step 2: comparing weights.}
By~\eqref{eq:profile} and~\eqref{eq:transfer-rho},
\(\abs{w_{ij}-w'_{ij}}\le\lambda(w_{ij}+w'_{ij}+1)\,4n\varepsilon/s\).
Summing over \(j\) and using~\eqref{eq:transfer-mass} for both tuples,
\begin{equation}\label{ineq:sumw}
\sum_j\abs{w_{ij}-w'_{ij}}
\le\lambda(W_i+W'_i+n)\frac{4n\varepsilon}{s}
\le12\log(n+1)\frac{\varepsilon}{s}\,\max\{W_i,W'_i\}.
\end{equation}
For nonnegative vectors \(w,w'\) with positive sums \(W,W'\) we have
\begin{equation}\label{ineq:algebra}
\sum_j\Bigl|\frac{w_j}{W}-\frac{w'_j}{W'}\Bigr|
\le\frac{2}{\max\{W,W'\}}\sum_j\abs{w_j-w'_j}.
\end{equation}
Indeed, both sides of~\eqref{ineq:algebra} are symmetric in \((w,W)\) and \((w',W')\), so we may
assume \(W\ge W'\). Writing
\(w_j/W-w'_j/W'=(w_j-w'_j)/W-w'_j(1/W'-1/W)\), we obtain
\[
\sum_j\Bigl|\frac{w_j}{W}-\frac{w'_j}{W'}\Bigr|
\le\frac1W\Bigl(\sum_j\abs{w_j-w'_j}+W-W'\Bigr)
\le\frac2W\sum_j\abs{w_j-w'_j}.
\]

Combining~\eqref{ineq:sumw} and~\eqref{ineq:algebra} yields
\begin{equation}\label{eq:transfer-probability}
\sum_j\abs{p_{ij}-p'_{ij}}\le24\log(n+1)\frac{\varepsilon}{s}.
\end{equation}

\smallskip\noindent
\textsc{Step 3: comparing averages.}
If \(w_{ij}>0\), then \(\norm{x_j-x_i}<(2n-1)s\)
by~\eqref{eq:transfer-support}. If \(w'_{ij}>0\), then likewise
\(\norm{x'_j-x'_i}<(2n-1)s'\), so
\(\norm{x_j-x_i}<(2n-1)s+2\varepsilon<2ns\). Since
\(\sum_j(p_{ij}-p'_{ij})=0\),
\[
T_s(x)_i-T_{s'}(x')_i=\sum_j(p_{ij}-p'_{ij})(x_j-x_i)
+\sum_jp'_{ij}(x_j-x'_j),
\]
and~\eqref{eq:transfer-probability} gives
\[
\norm{T_s(x)_i-T_{s'}(x')_i}
\le2ns\sum_j\abs{p_{ij}-p'_{ij}}+\varepsilon
\le\bigl(1+48n\log(n+1)\bigr)\varepsilon=C_n\varepsilon.
\qedhere
\]
\end{proof}

\begin{proof}[Proof of Theorem~\ref{thm:main}]
Set \(r_n(A)=A\) when \(|A|<n\).
If \(A=\{x_1,\ldots,x_n\}\), put
\[
r_n(A)=\{T_{\delta_n(A)}(x)_1,\ldots,T_{\delta_n(A)}(x)_n\}.
\]
This is independent of the enumeration. A pair realizing
\(\delta_n(A)\) collides by \eqref{eq:transfer-collision}, so the set $r_n(A)$
has at most \(n-1\) points. It lies in the convex hull of \(A\) and satisfies
\begin{equation}\label{eq:normed-movement}
d_H(r_n(A),A)\le(2n-1)\delta_n(A).
\end{equation}

Let \(\varepsilon=d_H(A,B)\). If $\max(\delta_n(A), \delta_n(B)) \le 2\varepsilon\), 
inequality~\eqref{eq:normed-movement} yields 
\[
d_H(r_n(A),r_n(B))\le(8n-3)\varepsilon\le C_n\varepsilon.
\]
Otherwise Lemma~\ref{lem:matching} gives enumerations of the two
\(n\)-point sets with matching distances at most \(\varepsilon\), and
\(|\delta_n(A)-\delta_n(B)|\le2\varepsilon\). Applying
\eqref{eq:transfer-lip} to these enumerations with \(s=\delta_n(A)\) and
\(s'=\delta_n(B)\) yields  $\Lip (r_n) \le C_n$.  

The $1$-Lipschitz property of the midpoint retraction \(X(2)\to X(1) \) is straightforward.
\end{proof}

The following is a corollary of Theorems~\ref{thm:direct} and~\ref{thm:main}. 

\begin{corollary}\label{cor:convex} 
Recall the constants $B_{n, k}$ from \eqref{eq:combined}. 
For $1\le k <n$, every nonempty convex subset $Y$ of any normed space 
admits a $B_{n, k}$-Lipschitz retraction $r_{n, k}\colon Y(n)\to Y(k)$ 
whose output lies in the convex hull of the input.
\end{corollary} 

\section{Clustering via barycenters: proof of Theorem~\ref{thm:barycenter-intro}}\label{sec:barycenter}

Let $\Pfin(X)$ denote the space of finitely supported probability measures on a
nonempty metric space $X$, equipped with the $1$-Wasserstein distance
\[
W_1(\mu,\nu)=\inf_{\pi\in\Pi(\mu,\nu)}
\int_{X\times X}d(x,y)\,d\pi(x,y),
\]
where $\Pi(\mu,\nu)$ is the set of couplings of $\mu,\nu$.
\begin{definition}\label{def:barycenter}
A \emph{$K$-Lipschitz barycenter} on $X$ is a map
$\beta\colon \Pfin(X)\to X$ such that
\begin{equation}\label{eq:barycenter}
\beta(\delta_x)=x,\qquad
d(\beta(\mu),\beta(\nu))\le K W_1(\mu,\nu).
\end{equation}
\end{definition}

Here $\delta_x$ is the Dirac probability measure at $x$; it is unrelated
to the separation $\delta_n$ from Definition~\ref{def:separation}. In the terminology of Mendel and Naor~\cite{MendelNaor},
the existence of such a map $\beta$ means that $X$ is
$W_1$-barycentric with constant $K$.

\begin{lemma}\label{lem:probability-convex}
For every metric space $X$, the space $(\Pfin(X),W_1)$ is isometric to a
convex subset of a real normed space.
\end{lemma}

\begin{proof}
Let $\mathcal M^0_{\mathrm{fin}}(X)$ be the vector space of finitely
supported signed measures of total mass zero, with the
Kantorovich--Rubinstein norm
\[
\|\eta\|_{\mathrm{KR}}
=\sup_{\Lip(f)\le1}\left|\int_X f\,d\eta\right|.
\]
Put $S=\supp\mu\cup\supp\nu$.
Every $1$-Lipschitz function $f\colon S\to\mathbb R$ extends to a $1$-Lipschitz function on $X$ by the McShane--Whitney theorem. 
Thus the supremum defining $\|\mu-\nu\|_{\mathrm{KR}}$ is unchanged if it is taken over $1$-Lipschitz functions on $S$. Kantorovich--Rubinstein duality
\cite[I.5, Case~5.16]{Villani}, applied to the finite metric space $S$, therefore gives
$\|\mu-\nu\|_{\mathrm{KR}}=W_1(\mu,\nu)$.

Fixing $o\in X$, the map
$\mu\mapsto\mu-\delta_o$ is an isometric affine embedding
of $\Pfin(X)$, with convex image.
\end{proof}

The completion of $\mathcal M^0_{\mathrm{fin}}(X)$ is the Arens--Eells
space of $X$, also known as its Lipschitz-free space; see~\cite{Weaver}.

\begin{theorem}\label{thm:barycenter}
If a metric space $X$ admits a $K$-Lipschitz barycenter on $\Pfin(X)$,
then, for every $1\le k<n$, there is a retraction
$r\colon X(n)\to X(k)$ such that
$\Lip(r)\le K B_{n,k}$, with $B_{n, k}$ as in~\eqref{eq:combined}. 
It can be chosen to satisfy
\begin{equation}\label{eq:barycentric-hull}
 r(A)\subset
 \{\beta(\mu):\mu\in\Pfin(X),\ \supp\mu\subset A\}.
\end{equation} 
\end{theorem}

\begin{proof}
By Lemma~\ref{lem:probability-convex} and Corollary~\ref{cor:convex},
there is a retraction
$S\colon \Pfin(X)(n)\to\Pfin(X)(k)$ with $\Lip(S)\le B_{n,k}$, 
whose output lies in the convex hull of its input. The map
$\iota\colon X(n)\to\Pfin(X)(n)$ given by $\iota(A)=\{\delta_x:x\in A\}$ 
is isometric, because $W_1(\delta_x,\delta_y)=d(x,y)$.
For $A\in X(n)$, define
$r(A)= \beta(S(\iota(A)))$. By construction, $r(A)$ has at most $k$ points. If $|A|\le k$, then
$S(\iota(A))=\iota(A)$ and $r(A)=A$.
The Lipschitz estimate follows by composition.
Every measure in $S(\iota(A))$ is a convex combination of the
Dirac measures at points of $A$, proving~\eqref{eq:barycentric-hull}.
\end{proof}

\begin{remark} \label{rem:two-points}
If \(X\) admits a \(K\)-Lipschitz barycenter $\beta$, then the map
\(r(\{x,y\})=\{\beta((\delta_x+\delta_y)/2)\}\) is easily seen to be a $K$-Lipschitz retraction $X(2)\to X(1)$. For the usual vector barycenter in a normed
space, this is the midpoint retraction.
\end{remark}

\begin{corollary} \label{cor:wasserstein}
For every metric space $M$, the space $(\Pfin(M),W_1)$ satisfies
$\Lambda_{n,k}\le B_{n,k}$ for all $1\le k<n$. 
If $M$ is Polish, the same conclusions hold
for the space $(\mathcal P_1(M),W_1)$ of Borel probability measures
with finite first moment.
\end{corollary}

\begin{proof}
The first assertion follows directly from
Lemma~\ref{lem:probability-convex} and Corollary~\ref{cor:convex}.
For the second, use the vector space of finite signed Borel measures
$\eta$ satisfying
\[
\eta(M)=0,\qquad
\int_M d(o,x)\,d|\eta|(x)<\infty,
\]
with the same Kantorovich--Rubinstein norm.
The duality formula for $\mathcal P_1(M)$~\cite[I.5, Case~5.16]{Villani} shows that
$\mu\mapsto\mu-\delta_o$ is an isometry with convex image.
Apply Corollary~\ref{cor:convex} again.
\end{proof}

For example, $(\mathcal P_1(S^1),W_1)$ has the Lipschitz clustering property, although $S^1$ itself
does not have it~\cite{KovalevLCP}. 

\begin{definition}\label{def:conical} ~\cite{Basso, DescombesLang} 
A \emph{conical geodesic bicombing} on a metric space $X$ assigns to
each ordered pair $x,y\in X$ a constant-speed geodesic
$\sigma_{xy}\colon [0,1]\to X$ from $x$ to $y$, such that
\begin{equation}\label{eq:conical}
d(\sigma_{xy}(t),\sigma_{x'y'}(t))
\le(1-t)d(x,x')+td(y,y')
\qquad(0\le t\le1).
\end{equation}
A subset of $X$ is \emph{$\sigma$-convex} if it contains the selected
geodesic between each pair of its points, and
$\overline{\conv}_{\sigma} A$ denotes the smallest closed
$\sigma$-convex subset containing $A$.
\end{definition}

If $X$ is complete and admits a conical bicombing $\sigma$, then it admits a
$1$-Lipschitz barycenter $\beta$ on $\Pfin(X)$ with
$\beta(\mu)\in\overline{\conv}_{\sigma}(\supp\mu)$. This is the finite-support part of \cite[Theorems 2.6 and 2.7]{Basso}, where the barycenter is constructed
on the space of Radon probability measures with finite first moment.

\begin{corollary}\label{cor:conical}
Every complete metric space \(X\) with a conical geodesic bicombing
satisfies $\Lambda_{n,k}(X)\le B_{n,k}$ for all $1\le k<n$.  
The retractions can be chosen so that
$r(A)\subset\overline{\conv}_{\sigma} A$.
In particular, this applies to every complete Busemann-convex space,
and hence to every Hadamard space. 
\end{corollary}

\begin{proof}
Apply Theorem~\ref{thm:barycenter} to the barycenter just described.
Its   property $\beta(\mu)\in\overline{\conv}_{\sigma}(\supp\mu)$ gives the stated
containment, and Remark~\ref{rem:two-points} gives the two-point constant.
In a Busemann-convex geodesic space, the distance between
two constant-speed geodesics is convex, so its unique geodesics
satisfy~\eqref{eq:conical}.  
\end{proof}

Theorem~\ref{thm:barycenter}, Remark~\ref{rem:two-points}, and
Corollary~\ref{cor:conical} prove Theorem~\ref{thm:barycenter-intro}.

\section{Absolute retracts, lower bounds, and open questions}\label{sec:alr}

The existence of retractions between finite subset spaces is inherited by Lipschitz retracts.

\begin{lemma}[{\cite[Lemma 3.3]{AkoforKovalev}}]\label{lem:retract-transfer}
Let $g\colon X\to Z$ and $f\colon Z\to X$ be Lipschitz maps with
$f\circ g=\operatorname{id}_X$. Then
\[
\Lambda_{n,k}(X)\le\Lip(f)\Lip(g)\,\Lambda_{n,k}(Z)\qquad(1\le k<n).
\]
\end{lemma}

\begin{proof}
Apply maps to finite sets pointwise. If $R\colon Z(n)\to Z(k)$ is a
retraction, then $f\circ R\circ g$ is a retraction $X(n)\to X(k)$, and
its Lipschitz constant is at most $\Lip(f)\Lip(R)\Lip(g)$.
\end{proof}

\begin{definition}\label{def:alr}
A metric space $X$ is an \emph{absolute $\lambda$-Lipschitz retract} if
every isometric copy of $X$ in a metric space is the image of a
$\lambda$-Lipschitz retraction. 
\end{definition}

For example, $\ell^\infty(\Gamma)$ is an absolute $1$-Lipschitz retract.

\begin{corollary}\label{cor:alr}
Every absolute $\lambda$-Lipschitz retract $X$ satisfies
$\Lambda_{n,k}(X)\le\lambda(2n-1)$ for $1\le k<n$.
\end{corollary}

\begin{proof}
For a fixed $o\in X$, the map $x\mapsto d(x,\cdot)-d(o,\cdot)$ embeds $X$
isometrically into $\ell^\infty(X)$, so $X$ is a $\lambda$-Lipschitz
retract of $\ell^\infty(X)$. Apply Theorem~\ref{thm:direct-infinity} and
Lemma~\ref{lem:retract-transfer}.
\end{proof}

The constants in Theorems~\ref{thm:direct}, \ref{thm:main},
and~\ref{thm:direct-infinity} grow with $n$, and some growth is
unavoidable.

\begin{proposition}\label{prop:lower}
Let $X$ contain an isometric copy of a nondegenerate interval.
Then $\Lambda_n(X)\ge n-3$ for all $n\ge2$.
\end{proposition}

\begin{proof}
For an interval $[0,L]$, this is~\cite[Theorem 5.1]{AkoforKovalev}. Let
$c\colon[0,L]\to X$ be an isometry. The map $f(x)=\min\{L,d(x,c(0))\}$ is
$1$-Lipschitz and satisfies $f\circ c=\operatorname{id}$, so
Lemma~\ref{lem:retract-transfer} gives
$\Lambda_n([0,L])\le\Lambda_n(X)$.
\end{proof}

\begin{question}\label{q:linear}
Is $\Lambda_n(X)=O(n)$ for every normed space $X$?
\end{question}

\begin{question}\label{q:uniform}
Is $\sup_{1\le k<n}\Lambda_{n,k}(X)=O(n)$ for every normed space $X$?
\end{question}

For normed spaces of dimension at least two, the lower bound also grows
with the size of the target: $\Lambda_{n,k}(X)\ge k/(2\pi)-1/2$
by~\cite[Theorem 1.1]{AkoforKovalev}.

\section*{AI use disclosure}

The author used ChatGPT (OpenAI) and Claude (Anthropic)
to assist with mathematical exploration and manuscript preparation.
The author reviewed and verified the mathematical arguments,
revised the manuscript, and takes full responsibility for its content.

\bibliographystyle{plain}
\bibliography{references_1}

\end{document}